\documentclass[preprint]{imsart}

\RequirePackage[OT1]{fontenc}
\RequirePackage{amsthm,amsmath,amssymb,bm}
\RequirePackage{natbib}
\RequirePackage[colorlinks,citecolor=blue,urlcolor=blue]{hyperref}
\usepackage{mathtools, enumitem}
\mathtoolsset{showonlyrefs} 
\usepackage{booktabs}
\newcommand{\bmX}{\bm{X}}
\newcommand{\bmx}{\bm{x}}
\newcommand{\bmth}{\bm{\theta}}

\newcommand{\bmSi}{\bm{\Sigma}}
\newcommand{\tran}{^{\mkern-1.5mu\mathsf{T}}}
\newcommand{\bmde}{\bm{\delta}}
\newcommand{\bmI}{\bm{I}}
\newcommand{\rd}{\mathrm{d}}
\newcommand{\bmzero}{\bm{0}}
\newcommand{\mymid}{\,|\,}
\newcommand{\bmG}{\bm{G}}
\DeclareMathOperator*{\argmax}{arg\,max}
\newcommand{\bmt}{\bm{t}}
\def\citeapos#1{\citeauthor{#1}'s (\citeyear{#1})}

\startlocaldefs
\numberwithin{equation}{section}
\theoremstyle{plain}
\newtheorem{thm}{Theorem}[section]

\newtheorem{corollary}{Corollary}[section]
\theoremstyle{remark}
\newtheorem{remark}{Remark}[section]
\theoremstyle{definition}

\endlocaldefs

\begin{document}

\begin{frontmatter}
\title{Ensemble minimaxity of James-Stein estimators}
\runtitle{Ensemble minimaxity}

\begin{aug}
\author{\fnms{Yuzo} \snm{Maruyama}%\thanksref{t1,t2}
\ead[label=e1]{maruyama@port.kobe-u.ac.jp}},
\author{\fnms{Lawrence D.} \snm{Brown}}
\and
\author{\fnms{Edward I.} \snm{George}%\thanksref{t2}
\ead[label=e2]{edgeorge@wharton.upenn.edu}}

\address{Kobe University and University of Pennsylvania \\
\printead{e1,e2}}

%\thankstext{t1}{Some comment}
%\thankstext{t2}{First supporter of the project}
%\thankstext{t3}{Second supporter of the project}
\runauthor{Y.~Maruyama et al.}

%\affiliation{Kobe University and University of Pennsylvania}

\end{aug}

\begin{abstract}
 This article discusses estimation of a multivariate normal mean based on heteroscedastic observations.
 Under heteroscedasticity, estimators shrinking more on the coordinates with larger variances,
 seem desirable. Although they are not necessarily minimax in the ordinary sense, we
show that such James-Stein type estimators can be {\itshape ensemble minimax},
 minimax with respect to the ensemble risk, related to empirical Bayes perspective of
 Efron and Morris.
\end{abstract}
\begin{keyword}[class=AMS]
\kwd[Primary ]{62C20}
\kwd[; secondary ]{62J07}
\end{keyword}

\begin{keyword}
\kwd{ensemble minimaxity}
\kwd{minimaxity}
\kwd{Stein estimation}
\end{keyword}

\end{frontmatter}

\section{Introduction}
Let $\bmX\sim N_p(\bmth,\bmSi)$ where $p\geq 3$, $\bmth=(\theta_1,\dots,\theta_p)\tran $ and
$\bmSi=\mathrm{diag}(\sigma^2_1,\dots,\sigma^2_p)$.
Let us assume 
\begin{equation}\label{sigma_descending}
 \sigma_1^2> \sigma^2_2 > \dots > \sigma^2_p.
\end{equation}
We are interested in the estimation of $\bmth$
with respect to the ordinary squared error loss function
\begin{equation}\label{ordinary_squared_error_loss}
 L(\bmde,\bmth)=\|\bmde-\bmth\|^2,
\end{equation}
where the risk of an estimator $\bmde(\bmX)$ is $ R(\bmde,\bmth)= E\left[ L(\bmde,\bmth)\right]$.
The MLE $\bmX$ with constant risk $\sum \sigma^2_i$ is shown to be extended Bayes
and hence minimax for any $p$ and any $\bmSi$.

In the homoscedastic case $ \sigma_1^2= \dots = \sigma^2_p$,
\cite{James-Stein-1961} showed that
the shrinkage estimator
\begin{equation}\label{JS-original}
 \left(1-\frac{c}{\bmX\tran \bmSi^{-1}\bmX}\right)\bmX\text{ for }c\in\left(0,2(p-2)\right)
\end{equation}
dominates the MLE $\bmX$ for $p\geq 3$.
There is some literature discussing the minimax properties of 
shrinkage estimators under heteroscedasticity.
\cite{Brown-1975}
showed that the James-Stein estimator \eqref{JS-original} is not necessarily minimax
when the variances are not equal.
Specifically, it is not minimax for any $c\in\left(0,2(p-2)\right)$
when $2\sigma^2_1>\sum_{i=1}^p\sigma^2_i$.
\cite{Berger-1976} showed that
\begin{equation}\label{JS-variant-1}
 \left(\bmI-\bmSi^{-1}\frac{c}{\bmX\tran \bmSi^{-2}\bmX}\right)\bmX\text{ for }c\in\left(0,2(p-2)\right)
\end{equation}
is minimax for $p\geq 3$ and any $\bmSi$. 
However, \cite{Casella-1980} argued that the James-Stein estimator \eqref{JS-variant-1}
may not be desirable even if it is minimax.
Ordinary minimax estimators, as in \eqref{JS-variant-1},
typically shrink most on the coordinates with smaller variances.
From \citeapos{Casella-1980} viewpoint, 
one of the most natural Jame-Stein variants is
\begin{equation}\label{JS-variant-2}
 \left(\bmI-\bmSi\frac{c}{\|\bmX\|^2}\right)\bmX\text{ for }c>0,
\end{equation}
which we are going to rescue, by providing some minimax properties related to Bayesian viewpoint.

In many applications, $\theta_i$ are thought to follow some exchangeable prior
distribution $\pi$. It is then natural to consider the compound risk function
which is then the Bayes risk with respect to the prior $\pi$
\begin{align}\label{eq:Bayes_risk_0}
 \bar{R}(\bmde,\pi)%=E_\pi(R(\bmth,\bmde))
=\int_{\mathbb{R}^p} R(\bmth,\bmde)\pi(\rd \bmth).
\end{align}
\cite{Efron-Morris-1971, Efron-Morris-1972-biometrika, Efron-Morris-1972-jasa, Efron-Morris-1973-jasa} addressed this problem from both the Bayes and empirical Bayes perspective.
In particular, they considered a prior distribution 
$\bmth\sim N_p(\bmzero,\tau\bmI_p)$ with
$\tau \in (0,\infty)$, and used the term  ``ensemble risk'' for the compound risk.
By introducing a set of ensemble risks
\begin{equation}\label{eq:Bayes_risk}
 \bar{R}(\bmde,\tau)=\int_{\mathbb{R}^p} R(\bmde,\bmth)\frac{1}{(2\pi\tau)^{p/2}}
  \exp\left(-\frac{\|\bmth\|^2}{2\tau}\right)\rd \bmth,
\end{equation}
we can define ensemble minimaxity with respect to a set of priors
\begin{equation}\label{P_star}
 \mathcal{P}_\star =\{N_p(\bm{0},\tau \bm{I}_p):\tau\in(0,\infty)\},
\end{equation}
that is, an estimator $\bmde$ is said to be ensemble minimax with respect to $\mathcal{P}_\star$ if
\begin{equation}\label{em_P_*}
 \sup_{\tau\in(0,\infty)}\bar{R}(\bmde,\tau)=
 \inf_{\bmde'}\sup_{\tau\in(0,\infty)}\bar{R}(\bmde',\tau).
\end{equation}
As a matter of fact, the second author in his unpublished manuscript,
\cite{Brown-ensemble-2011}, has already introduced the concept of ensemble minimaxity.
In this article, we follow their spirit but propose a simpler and clearer approach for 
establishing ensemble minimaxity of estimators.

Our article is organized as follows. In Section \ref{sec:em},
we elaborate the definition of ensemble minimaxity and
explain \citeapos{Casella-1980} viewpoint on the contradiction between minimaxity and
well-conditioning.
In Section \ref{sec:main}, we 
show the ensemble minimaxity of various shrinkage estimators
including a variant of the James-Stein estimator
\begin{equation}\label{intro.eq:nice_js}
 \left(\bmI-\bmSi\frac{p-2}{(p-2)\sigma^2_1+\|\bmX\|^2}\right)\bmX
\end{equation}
as well as the generalized Bayes estimator with respect to
the hierarchical prior
\begin{equation}\label{eq:gharmonic_intro}
 \bm{\theta}\mymid\lambda \sim N_p(\bm{0},(\sigma^2_1/\lambda)\bm{I}-\bm{\Sigma}), \ \pi(\lambda) \sim
\lambda^{-2}I_{(0,1)}(\lambda)
\end{equation}
which is a generalization of the harmonic prior $\|\bmth\|^{2-p}$ for the heteroscedastic case.

\section{Minimaxity, Ensemble Minimaxity and Casella's viewpoint}
\label{sec:em}
If the prior $\pi(\bmth)$ were known,
the resulting posterior mean $E[\bmth\mymid \bmx]$ would then be the
optimal estimate under the sum of the squared error loss.
However, it is typically not feasible to exactly specify the prior.
One approach to avoid excessive dependence on the choice of prior,
is to consider a set of priors $\mathcal{P}$ on $\Theta$ and 
study the properties of estimators based on the corresponding set
of ensemble risks.
As in classical decision theory, there rarely exists an
estimator that achieves the minimum ensemble risk uniformly for all $\pi\in\mathcal{P}$.
A more realistic goal as pursued in this paper is to study the ensemble
minimaxity of James-Stein type estimators.

Recall that with ordinary risk $R(\bmde,\bmth)$,
$\bmde$ is said to be minimax if 
\begin{align}
 \sup_{\bmth\in\Theta}R(\bmde,\bmth)
 =\inf_{\bmde'}\sup_{\bmth\in\Theta}R(\bmde',\bmth).
\end{align}
Similarly for the case of ensemble risk we have the following definition.
Note the Bayes risk of $\bmde$ under the prior $\pi$ is given by \eqref{eq:Bayes_risk_0}.
The estimator $\bmde$ is said to be ensemble minimax with respect to $\mathcal{P}$ if 
\begin{align}
 \sup_{\pi\in\mathcal{P}}\bar{R}(\bmde,\pi)=
 \inf_{\bmde'}\sup_{\pi\in\mathcal{P}}\bar{R}(\bmde',\pi).
\end{align}

The motivation for the above definitions comes from the use of the empirical Bayes method
in simultaneous inference. \cite{Efron-Morris-1972-jasa}, 
derived the James-Stein estimator through the parametric empirical Bayes model with
$\bmth\sim N_p(\bmzero,\tau \bmI_p)$.
Note that in such an empirical Bayes model, $\tau$ is the unknown non-random parameter.
Given the family $\mathcal{P}_\star =\{N_p(\bm{0},\tau \bm{I}_p):\tau\in(0,\infty)\}$,
%ensemble minimaxity with respect to $\mathcal{P}_\star$, defined by \eqref{em_P_*},
%Also, with the family $\mathcal{P}_\star$, 
the Bayes risk is a function of $\tau$ as follows,
\begin{equation}\label{eq:Bayes_risk_1}
 \bar{R}(\bmde,\tau)=\int_{\mathbb{R}^p} R(\bmde,\bmth)\frac{1}{(2\pi\tau)^{p/2}}
  \exp\left(-\frac{\|\bmth\|^2}{2\tau}\right)\rd \bmth.
\end{equation}
Hence, with $ \bar{R}(\bmde,\tau)$, the estimator $\bmde$ is said to be ensemble minimax with respect to $\mathcal{P}_\star$ if
\begin{equation}\label{em_P_*_1}
 \sup_{\tau\in(0,\infty)}\bar{R}(\bmde,\tau)=
 \inf_{\bmde'}\sup_{\tau\in(0,\infty)}\bar{R}(\bmde',\tau),
\end{equation}
which may be seen as the counterpart of ordinary minimaxity in the empirical Bayes model.

Clearly the usual estimator $\bmX$ has constant risk, has constant Bayes risk and hence
% $\sum_{i=1}^p\sigma^2_i$
$\bmX$ is ensemble minimax. Then the ensemble minimaxity of $\bmde$ follows if
\begin{align*}
 \bar{R}(\bmde,\tau)\leq \sum_{i=1}^p\sigma^2_i, \ \forall \tau\in (0,\infty).
\end{align*}
%for all $\tau\in(0,\infty)$.

\begin{remark}
Note that ensemble minimaxity can also be interpreted as a particular
case of Gamma minimaxity studied in the context of robust Bayes analysis
by \cite{Good-1952, Berger_L-1979}.
However, in such studies, a ``large'' set consisting of many diffuse priors
are usually included in the analysis.
Since this is quite different from our formulation of the problem,
we use the term ensemble minimaxity throughout our paper,
following the Efron and Morris papers cited above.
\end{remark}

\medskip

A class of shrinkage estimators which we consider in this paper, is given by
\begin{equation}\label{phiphiphi}
\bmde_{\phi}= \left(\bmI -\bmG
	       \frac{\phi(z)}{z}\right)\bmx, \
\mbox{ for }z=\bmx\tran \bm{G}\bmSi^{-1}\bmx=\sum \frac{g_ix_i^2}{\sigma^2_i},
\end{equation}
where $\bm{G}=\mbox{diag}(g_1,\dots,g_p)$ with
\begin{equation*}
 0<g_i\leq 1, \ \forall i.
\end{equation*}
\cite{Berger-Srinivasan-1978} showed, in their Corollary 2.7,
that, given positive-definite $\bm{C}$ and non-singular $\bm{B}$,
a necessary condition for an estimator of the form
\begin{align*}
 \left(\bmI -\bm{B}\frac{\phi(\bmx\tran \bm{C}\bmx)}{\bmx\tran \bm{C}\bmx}\right)\bmx 
\end{align*}
to be admissible is $\bm{B}= k\bmSi\bm{C}$ for some constant $k$,
which is satisfied by estimators among the class of \eqref{phiphiphi}.

A version of \citeapos{Baranchik-1964} sufficient condition for ordinary minimaxity
is given in Appendix \ref{sec:ordinary}; For given $\bmG$ which satisfies
\begin{align*}
 h(\bmSi,\bmG)=2\left(\frac{\sum g_i\sigma^2_i}{\max (g_i\sigma^2_i)}-2\right)>0,
\end{align*}
$\bmde_{\phi}$ given by \eqref{phiphiphi} is ordinary minimax if
\begin{equation}\label{suffi.condi.ordi.minimax}
 \phi(\cdot) \text{ is non-decreasing and }0\leq \phi \leq h(\bmSi,\bmG).
\end{equation}
\cite{Berger-1976} showed that, for any given $\bmSi$,
\begin{align*}
 \max_{\bmG}h(\bmSi,\bmG)=2(p-2),\quad \argmax_{\bmG}h(\bmSi,\bmG)=\sigma^2_p\bmSi^{-1}=\mathrm{diag}\left(\frac{\sigma^2_p}{\sigma_1^2}, \dots, \frac{\sigma^2_p}{\sigma_{p-1}^2},1\right)
\end{align*}
which seems the right choice of $\bmG$.
However, from the ``conditioning'' viewpoint of \cite{Casella-1980}
which advocates more shrinkage on higher variance estimates, the descending order
\begin{equation}\label{eq:ascend}
 g_1>\dots >g_p
\end{equation}
is desirable, whereas
$ \bmG=\sigma^2_p\bmSi^{-1}$ corresponding to the ascending order
$g_1<\dots <g_p$ under $\bmSi$ given by \eqref{sigma_descending}.
As \cite{Casella-1980} pointed out,
ordinary minimaxity cannot be enjoyed together with well-conditioning given by \eqref{eq:ascend}
when
\begin{equation*}
 h(\bmSi,c\bmI)\leq 0 \text{ or equivalently }\sum\sigma^2_i\leq 2\sigma^2_1
\end{equation*}
for some $0<c\leq 1$.
In fact, when $  h(\bmSi,c\bmI)\leq 0$ and $ c=g_1>\dots >g_p$, % as in \eqref{eq:ascend},
we have
\begin{align*}
 c\sigma^2_1=g_1\sigma^2_1, \ 
 c\sigma^2_2>g_2\sigma^2_2, \ \dots, \ c\sigma^2_p>g_p\sigma^2_p
\end{align*}
and hence $ h(\bmSi,\bmG)<0$ follows.
The motivation of \cite{Casella-1980,Casella-1985} seems to provide a better
treatment for the case.
Actually \cite{Brown-1975} pointed out
essentially the same phenomenon from a slightly different viewpoint.

Ensemble minimaxity, based on ensemble risk given by \eqref{eq:Bayes_risk},
provides a way of saving shrinkage estimators with well-conditioning,
estimators which are not necessarily ordinary minimax.

 \section{Ensemble minimaxity}
\label{sec:main}
\subsection{A general theorem}
\label{subsec:general}
We have the following theorem on ensemble minimaxity of $ \bmde_{\phi}$ with general $\bmG$, though
we will eventually focus on $ \bmde_{\phi}$ with $\bmG$ with 
the descending order $g_1>\dots >g_p$ as in \eqref{eq:ascend}. 
\begin{thm}
  Assume $\phi(z)$ is non-negative, non-decreasing and concave.
  Also $\phi(z)/z$ is assumed non-increasing.
  Then
\begin{equation*}
\bmde_{\phi}= \left(\bmI -\bmG
	       \frac{\phi(z)}{z}\right)\bmx, \
\mbox{ for }z=\bmx\tran \bm{G}\bmSi^{-1}\bmx=\sum \frac{g_ix_i^2}{\sigma^2_i}
\end{equation*}
  is ensemble minimax if
\begin{equation}\label{AAA}
\phi(p\min_ig_i(1+\tau/\sigma^2_i)) \leq 2(p-2)
\frac{\min_ig_i(1+\tau/\sigma^2_i)}{\max_ig_i(1+\tau/\sigma^2_i)}, \quad \forall \tau\in(0,\infty).
\end{equation}
\end{thm}
\begin{proof}
Recall, for $i=1,\dots,p$,
\begin{align*}
x_i \mymid \theta_i \sim N(\theta_i,\sigma^2_i),\text{ and }\theta_i\sim N(0,\tau).
\end{align*}
Then the posterior and marginal are given by
\begin{align*}
 \theta_i\mymid x_i\sim N\left(\frac{\tau}{\tau+\sigma^2_i}x_i,\frac{\tau\sigma^2_i}{\tau+\sigma^2_i}\right)
\text{ and }x_i\sim N(0,\tau+\sigma^2_i),
\end{align*}
respectively, where $ \theta_1\mymid x_1,\dots,\theta_p\mymid x_p$ are mutually independent and
$x_1,\dots,x_p$ are mutually independent. Then the Bayes risk is given by
\begin{align*}
 \bar{R}(\bmde_{\phi},\tau)&=
 \sum_{i=1}^p E_{\bmth} E_{\bmx\mymid \bmth}
 \left[\left\{\left(1-g_i\frac{\phi(z)}{z}\right)x_i-\theta_i\right\}^2\right] \\
&= \sum_{i=1}^p E_{\bmx} E_{\bmth\mymid \bmx}
 \left[\left\{\left(1-g_i\frac{\phi(z)}{z}\right)x_i-\theta_i\right\}^2\right] \\
&= \sum_{i=1}^p E_{\bmx} E_{\bmth\mymid \bmx}
 \left[\left\{\left(1-g_i\frac{\phi(z)}{z}\right)x_i- E[\theta_i\mymid x_i]+E[\theta_i\mymid x_i]-\theta_i\right\}^2\right] \\
 &= \sum_{i=1}^p E_{\bmx} \left[\left\{\left(1-g_i\frac{\phi(z)}{z}\right)x_i- E[\theta_i\mymid x_i]\right\}^2\right]
 +\sum_{i=1}^p\mathrm{Var}(\theta_i\mymid x_i) \\
 &= \sum_{i=1}^p E_{\bmx} \left[\left(\frac{\sigma^2_i}{\tau+\sigma^2_i}x_i -
 g_i\frac{\phi(z)}{z}x_i\right)^2\right]
 +\sum_{i=1}^p\frac{\tau\sigma^2_i}{\tau+\sigma^2_i}.
\end{align*}
Since the first term of the r.h.s.~of the above equality is rewritten as 
\begin{align*}
& \sum_{i=1}^p E_{\bmx} \left[\left(\frac{\sigma^2_i}{\tau+\sigma^2_i}x_i -
 g_i\frac{\phi(z)}{z}x_i\right)^2\right] \\
&= \sum_{i=1}^p \left(\frac{\sigma^2_i}{\tau+\sigma^2_i}\right)^2E_{\bmx}[x_i^2]
-2E_{\bmx} \left[\sum_{i=1}^p\frac{\sigma^2_ig_ix_i^2}{\tau+\sigma^2_i}
 \frac{\phi(z)}{z}\right]+
 E_{\bmx} \left[\sum_{i=1}^pg^2_ix_i^2\frac{\phi^2(z)}{z^2}\right] \\
 &=\sum_{i=1}^p\frac{\sigma^4_i}{\tau+\sigma^2_i}
-2E_{\bmx} \left[\sum_{i=1}^p\frac{\sigma^2_ig_ix_i^2}{\tau+\sigma^2_i}
 \frac{\phi(z)}{z}\right]+
 E_{\bmx} \left[\sum_{i=1}^pg^2_ix_i^2\frac{\phi^2(z)}{z^2}\right],
\end{align*}
 we have
\begin{equation}\label{eq:bayes_risk_diff}
 \bar{R}(\bmde_{\phi},\tau)-\sum \sigma^2_i 
= -2E_{\bmx} \left[\sum_{i=1}^p\frac{\sigma^2_ig_ix_i^2}{\tau+\sigma^2_i}
 \frac{\phi(z)}{z}\right]+
 E_{\bmx} \left[\sum_{i=1}^pg^2_ix_i^2\frac{\phi^2(z)}{z^2}\right].
\end{equation} 
 Let 
\begin{align*}
 w_i=\frac{x_i^2}{\sigma^2_i+\tau},\  w=\sum_{i=1}^p w_i \ \text{ and } \ t_i=\frac{w_i}{w}\text{ for }i=1,\dots,p.
\end{align*}
%$w_i=x_i^2/(\sigma^2_i+\tau)$, $w=\sum_{i=1}^p w_i$ and $t_i=w_i/w$ for $i=1,\dots,p$.
Then
\begin{gather*}
w\sim\chi^2_p, \quad
\bmt=(t_1,\dots,t_p)\tran   \sim \mathrm{Dirichlet}(1/2,\dots,1/2),
\end{gather*}
and $w$ and $\bmt$ are mutually independent.
With the notation, we have
\begin{align*}
 x_i^2=wt_i(\sigma^2_i+\tau)\text{ and }z=
\bm{x}\tran \bm{G}\bm{\Sigma}^{-1}\bm{x}=
\sum_{i=1}^p \frac{g_ix_i^2}{\sigma^2_i}=w\sum_{i=1}^p t_ig_i\left(1+\frac{\tau}{\sigma^2_i}\right)
\end{align*}
and hence
\begin{align*}
 E_{\bmx} \left[\sum g^2_ix_i^2\frac{\phi^2(z)}{z^2}\right] 
 &=E_{w,\bmt} \left[\frac{\sum t_ig^2_i(\sigma^2_i+\tau)}
 {\sum t_ig_i(1+\tau/\sigma^2_i)}
\frac{\phi(w\sum t_ig_i(1+\tau/\sigma^2_i))^2}{w\sum t_ig_i(1+\tau/\sigma^2_i)}
 \right] \\
 &=E_{\bmt} \left[\frac{\sum t_ig^2_i(\sigma^2_i+\tau)}
 {\sum t_ig_i(1+\tau/\sigma^2_i)}
E_{w\mymid \bmt}\left[\frac{\phi(w\sum t_ig_i(1+\tau/\sigma^2_i))^2}{w\sum t_ig_i(1+\tau/\sigma^2_i)}
 \right] \right].
 \end{align*}
 Since $\phi(w)/w$ is non-increasing and $\phi(w)$ is non-decreasing,
 by the correlation inequality, we have
 \begin{align*}
  &  E_{w\mymid \bmt}\left[
  \frac{\phi(w\sum t_ig_i(1+\tau/\sigma^2_i))^2}{w\sum t_ig_i(1+\tau/\sigma^2_i)}\right] \\
&\leq   
 E_{w\mymid \bmt}\left[\phi(w\sum t_ig_i(1+\tau/\sigma^2_i))\right]
 E_{w\mymid \bmt}\left[\frac{\phi(w\sum t_ig_i(1+\tau/\sigma^2_i))}{w\sum t_ig_i(1+\tau/\sigma^2_i)}\right] \\
&\leq  E_{w\mymid \bmt}\left[\phi(w\sum t_ig_i(1+\tau/\sigma^2_i))\right]
 E_{w}\left[
 \frac{\phi(w\min g_i(1+\tau/\sigma^2_i))}{w \min g_i(1+\tau/\sigma^2_i)}
 \right],
 \end{align*}
 and hence
 \begin{equation}\label{inequ.proof}
\begin{split}
 &   E_{\bmx} \left[\sum g^2_ix_i^2\frac{\phi^2(z)}{z^2}\right]  
\leq E_{w}\left[
 \frac{\phi(w\min g_i(1+\tau/\sigma^2_i))}{w \min g_i(1+\tau/\sigma^2_i)}
  \right] \\
&\qquad \qquad \qquad \times E_{\bmt} \left[\frac{\sum t_ig^2_i(\sigma^2_i+\tau)}
 {\sum t_ig_i(1+\tau/\sigma^2_i)}
E_{w\mymid \bmt}\left[\phi(w\sum t_ig_i(1+\tau/\sigma^2_i)) \right] \right].  
\end{split}
 \end{equation}
 In the first part of the r.h.s.~of the inequality \eqref{inequ.proof}, we have
\begin{equation}\label{eq:first}
\begin{split}
E_{w}\left[ \frac{\phi(w\min g_i(1+\tau/\sigma^2_i))}{w} \right]
 &
 \leq E_w[1/w]E_w\left[\phi(w\min_i g_i(1+\tau/\sigma^2_i))\right] \\%\text{ (correlation inequality)}\\
&\leq \frac{\phi(E_w\left[w\right]\min_i g_i(1+\tau/\sigma^2_i))}{p-2} \\%\text{ (Jensen's inequality)}\\
&= \frac{\phi(p\min_i g_i(1+\tau/\sigma^2_i))}{p-2},
\end{split}
\end{equation}
where the first and second inequality follow from the correlation inequality and Jensen's inequality, respectively.
 In the second part of the r.h.s.~of the inequality \eqref{inequ.proof},
by the inequality
\begin{align*}
 \sum t_ig^2_i(\sigma^2_i+\tau)\leq \max_i g_i(1+\tau/\sigma_i^2) \sum t_ig_i\sigma^2_i,
\end{align*}
we have
 \begin{equation}\label{eq:second}
\begin{split}
 & E_{\bmt} \left[\frac{\sum t_ig^2_i(\sigma^2_i+\tau)}
 {\sum t_ig_i(1+\tau/\sigma^2_i)}
 E_{w\mymid \bmt}\left[\phi(w\sum t_ig_i(1+\tau/\sigma^2_i))\right]\right] \\
& = E_{w,\bmt}\left[\frac{\sum t_ig^2_i(\sigma^2_i+\tau)}
 {\sum t_ig_i(1+\tau/\sigma^2_i)}
\phi(w\sum t_ig_i(1+\tau/\sigma^2_i))\right]\\
 & \leq \max_i g_i(1+\tau/\sigma_i^2)
 E_{w,\bmt}\left[\frac{\sum t_ig_i\sigma^2_i}
 {\sum t_ig_i(1+\tau/\sigma^2_i)}
 \phi(w\sum t_ig_i(1+\tau/\sigma^2_i))\right] \\
 &=\max_i g_i(1+\tau/\sigma_i^2)
E_{\bmx} \left[\sum\frac{\sigma^2_ig_ix_i^2}{\tau+\sigma^2_i}
 \frac{\phi(z)}{z}\right].
\end{split}
 \end{equation}
By \eqref{inequ.proof}, \eqref{eq:first} and \eqref{eq:second}, we have
\begin{equation}\label{eq:-1}
 \begin{split}
&E_{\bmx} \left[\sum g^2_ix_i^2\frac{\phi^2(z)}{z^2}\right] \\  
  &\leq \frac{\phi(p\min_i g_i(1+\tau/\sigma^2_i))}{p-2}
  \frac{\max_i g_i(1+\tau/\sigma^2_i)}{\min_i g_i(1+\tau/\sigma^2_i)}
E_{\bmx} \left[\sum\frac{\sigma^2_ig_ix_i^2}{\tau+\sigma^2_i}
 \frac{\phi(z)}{z}\right],  
 \end{split}
\end{equation} 
and, by \eqref{eq:bayes_risk_diff} and \eqref{eq:-1},
\begin{align*}
& \bar{R}(\bmde_{\phi},\tau)-\sum \sigma^2_i \\
&\leq 
\left(\frac{\phi(p\min_i g_i(1+\tau/\sigma^2_i))}{p-2}
  \frac{\max_i g_i(1+\tau/\sigma^2_i)}{\min_i g_i(1+\tau/\sigma^2_i)}-2\right)
 E_{\bmx} \left[\sum\frac{\sigma^2_ig_ix_i^2}{\tau+\sigma^2_i}
 \frac{\phi(z)}{z}\right],
\end{align*}
 which guarantees $ \bar{R}(\bmde_{\phi},\tau)-\sum \sigma^2_i \leq 0$ for all $\tau\in(0,\infty)$
 under the condition \eqref{AAA}.
\end{proof}

Given $\bmSi$, the choice $\bm{G}=\bm{\Sigma}/\sigma^2_1$ 
with descending order $g_1>\dots>g_p$,
is one of the most natural choice of $\bmG$ from \citeapos{Casella-1980} viewpoint.
In this case,
we have
\begin{gather*}
\frac{\min_ig_i(1+\tau/\sigma^2_i)}{\max_ig_i(1+\tau/\sigma^2_i)}
=\frac{\min_i \{\sigma_i^2+\tau\}}
{\max_i\{\sigma_i^2+\tau\}}
=\frac{\sigma_p^2+\tau}{\sigma_1^2+\tau}, \\
 p\min_ig_i(1+\tau/\sigma^2_i)=\frac{p}{\sigma^2_1}\min_i(\sigma^2_i+\tau)=p\frac{\sigma^2_p+\tau}
 {\sigma^2_1},
\end{gather*}
and hence
a following corollary.
\begin{corollary}\label{cor:1}
Assume that 
$\phi(z)$ is non-negative, non-decreasing and concave. Also  $\phi(z)/z$ is
assumed non-increasing.
 Then
\begin{align*}
 \bm{\delta}_\phi=\left(\bmI-\bmSi\frac{\phi(\|\bmx\|^2/\sigma^2_1)}{\|\bmx\|^2}\right)\bmx
\end{align*}
 is ensemble minimax if 
\begin{equation}\label{simple}
\phi(p(\sigma^2_p+\tau)/\sigma^2_1) \leq 2(p-2)
\frac{\sigma^2_p+\tau}{\sigma^2_1+\tau}\quad, \forall \tau\in(0,\infty).
\end{equation}
\end{corollary} 

  \subsection{An ensemble minimax James-Stein variant}
\label{subsec:js}
As an example of Corollary \ref{cor:1}, we consider
\begin{align}\label{eq:form_stein}
 \phi(z)=\frac{c_1z}{c_2+z}
\end{align}
for $c_1>0$ and $c_2\geq 0$,
which is motivated by \cite{Stein-1956} and \cite{James-Stein-1961}.
Under $\bmSi=\bmI_p$, \cite{Stein-1956} suggested that
there exist estimators dominating the usual estimator $\bmx$
among a class of estimators $\bmde_\phi$ with $\phi$ given by \eqref{eq:form_stein}
for small $c_1$ and large $c_2$.
Following \cite{Stein-1956}, \cite{James-Stein-1961} showed that
$ \bmde_\phi $ with $0<c_1<2(p-2)$ and  $ c_2=0 $ is ordinary minimax.
The choice $ c_2=0 $ is, however, not good since, by Corollary \ref{cor:1}, $c_1$ cannot be larger than
$2(p-2)\sigma^2_p/\sigma^2_1$. With positive $c_2$, we can see that $c_1$ can be much larger as follows.

Note that $\phi(z)$ given by \eqref{eq:form_stein}
is non-negative, increasing and concave and that $ \phi(z)/z$ is decreasing.
Then the sufficient condition in \eqref{simple}
is
\begin{align*}
\frac{c_1p(\sigma^2_p+\tau)/\sigma^2_1}{c_2+p(\sigma^2_p+\tau)/\sigma^2_1} \leq 2(p-2)
\frac{\sigma^2_p+\tau}{\sigma^2_1+\tau}\quad \forall \tau\in(0,\infty),
\end{align*}
which is equivalent to
\begin{align*}
2(p-2)\left\{\sigma^2_1c_2+p(\sigma^2_p+\tau)\right\} -c_1p(\sigma^2_1+\tau)\geq 0
 \quad \forall \tau\in(0,\infty)
\end{align*}
or
\begin{align*}
 p\tau\left\{2(p-2)-c_1\right\}+2(p-2)\sigma^2_1\left\{c_2-p\left(\frac{c_1}{2(p-2)}-\frac{\sigma^2_p}{\sigma^2_1}\right)\right\}\geq 0
 \quad \forall \tau\in(0,\infty).
\end{align*}
Hence we have a following result.
\begin{thm}\label{cor:js}
 \begin{enumerate}
  \item \label{cor:js.1}
When
\begin{equation}\label{eq:cor:js.1}
 0<c_1\leq 2(p-2) \text{ and }c_2\geq \max\left(0, p\left(\frac{c_1}{2(p-2)}-\frac{\sigma^2_p}{\sigma^2_1}\right)\right),
\end{equation}
the shrinkage estimator
\begin{align*}
\left(\bmI-\bmSi\frac{c_1}{c_2\sigma^2_1+\|\bmx\|^2}\right)\bmx 
\end{align*}
is ensemble minimax.
  \item \label{cor:js.2}
	It is ordinary minimax if 
\begin{equation*}
 2\left(\sum \sigma_i^4/\sigma_1^4-2\right)\geq c_1.
\end{equation*}
 \end{enumerate}
\end{thm}
Part \ref{cor:js.2} above follows from Theorem \ref{thm:ordinary_minimax}.
%Remark: In Part \ref{cor:js.1}, some positive $c_2$ is needed for ensemble minimaxity
%when $\sigma^2_1>>\sigma^2_p$ and/or $c_1$ is relatively large.
%
%\medskip

It seems to us that one of the most interesting estimators with ensemble minimaxity
from Part \ref{cor:js.1}
is
\begin{equation}\label{eq:nice_js}
 \left(\bmI-\bmSi\frac{p-2}{(p-2)\sigma^2_1+\|\bmx\|^2}\right)\bmx 
\end{equation}
with the choice $c_1=c_2=p-2$ satisfying \eqref{eq:cor:js.1}.
It is clear that the $i$-th shrinkage factor
\begin{align*}
 1-\frac{(p-2)\sigma^2_i}{(p-2)\sigma^2_1+\|\bmx\|^2}
\end{align*}
is nonnegative for any $\bmx$ and any $\bmSi$, which is a nice property.

\subsection{A generalized Bayes ensemble minimax estimator}
\label{sec:Bayes}
In this subsection, we provide a generalized Bayes ensemble minimax estimator. Following \cite{Strawderman-1971}, \cite{Berger-1976} and
\cite{Maru-Straw-2005}, 
we consider the generalized harmonic prior
\begin{equation}\label{eq:gharmonic}
 \bm{\theta}\mymid \lambda \sim N_p(\bm{0},\lambda^{-1}\bm{\Sigma}\bm{G}^{-1}-\bm{\Sigma}), \ \pi(\lambda) \sim
\lambda^{-2}I_{(0,1)}(\lambda)
\end{equation}
where $\bmG =\mbox{diag}(g_1,\dots,g_p)$ satisfies %\ref{as.2}, that is,
$0<\forall g_i \leq 1$.
Note that for $\bmSi=\bmG=\bmI_p$, the density of $\bmth$ is exactly $\pi(\bmth)=\|\bmth\|^{2-p}$,
since $ \lambda^{-1}\bm{\Sigma}\bm{G}^{-1}-\bm{\Sigma}=\{(1-\lambda)/\lambda\}\bmI_p$ and
\begin{align*}
& \frac{1}{(2\pi)^{p/2}} \int_0^1\left(\frac{\lambda}{1-\lambda}\right)^{p/2}\exp\left(-\frac{\lambda\|\bmth\|^2}{2(1-\lambda)}\right) \lambda^{-2}\rd \lambda \\
 &=\frac{1}{(2\pi)^{p/2}} \int_0^\infty g^{p/2-2}\exp\left(-g\|\bmth\|^2/2\right)\rd g \\
&=\frac{\Gamma(p/2-1)2^{p/2-1}}{(2\pi)^{p/2}}\|\bmth\|^{2-p}.
\end{align*}
The prior $\pi(\bmth)=\|\bmth\|^{2-p}$ is called the harmonic prior and
was originally investigated by \cite{Baranchik-1964} and \cite{Stein-1974}.
\cite{Berger-1980} and \cite{Berger-Strawderman-1996} recommended the use of the prior
\eqref{eq:gharmonic} mainly because it is on the boundary of admissibility.

By the way of \cite{Strawderman-1971},
the generalized Bayes estimator with respect to the prior is given by 
\begin{equation}
\bm{\delta}_{*}= \left(\bmI-\bmG
\frac{\phi_*(z)}{z}\right)\bmx, \ \mbox{ for }z=\bmx\tran \bmG\bmSi^{-1}\bmx
\end{equation}
with
\begin{equation*}
 \phi_*(z)=z\frac{\int_0^1 \lambda^{p/2-1}\exp(-z\lambda/2)\rd\lambda }
{\int_0^1 \lambda^{p/2-2}\exp(-z\lambda/2)\rd\lambda },
\end{equation*}
where $\phi_*(z)$ satisfies the following properties
 \begin{enumerate}[label= H\arabic*]
  \item\label{Bayes_phi_1}
       $\phi_*(z)$ is increasing in $z$.
  \item\label{Bayes_phi_2}
       $\phi_*(z)$ is concave.
  \item \label{Bayes_phi_3}
	$\lim_{z\to\infty}\phi_*(z)=p-2$.
  \item\label{Bayes_phi_4}
       $\phi_*(z)/z$ is decreasing in $z$.
  \item \label{Bayes_phi_5}
	The derivative of $\phi_*(z)$ at $z=0$ is $(p-2)/p$.
 \end{enumerate}
Under the choice 
$\bm{G}=\bm{\Sigma}/\sigma_1^2$ and with the condition of
Corollary \ref{cor:1}, we have a following result.
\begin{thm}\label{final.thm}
\begin{enumerate}
\item \label{final.thm.1}
The estimator $\bm{\delta}_{*}$ is ensemble minimax.
 \item \label{final.thm.2}
The estimator        $\bm{\delta}_{*}$ is ordinary minimax when
\begin{equation*}
 2\left(\sum \sigma_i^4/\sigma_1^4-2\right)\geq p-2.
\end{equation*}
 \item \label{final.thm.3}
The estimator        $\bm{\delta}_{*}$ is conventional admissible.
\end{enumerate}
\end{thm} 
\begin{proof}\mbox{}
[Part \ref{final.thm.1}] 
Recall that the sufficient condition for ensemble minimaxity is given by
%in \eqref{simple} of 
Corollary \ref{cor:1}.
By \ref{Bayes_phi_1}--\ref{Bayes_phi_5}, we have only to check \eqref{simple} in Corollary \ref{cor:1}.

For $\tau\geq \max(0,\sigma^2_1-2\sigma^2_p)$, we have
\begin{align*}
2(p-2)\frac{\sigma_p^2+\tau}{\sigma_1^2+\tau}\geq  p-2. 
\end{align*}
 By the properties \ref{Bayes_phi_1} and \ref{Bayes_phi_3},
\begin{align*}
 \phi_*(p(\sigma^2_p+\tau)/\sigma^2_1) \leq p-2.
\end{align*}
for $\tau\in(0,\infty)$. Hence for $\tau\geq \max(0,\sigma^2_1-2\sigma^2_p)$, it follows that
\begin{align*}
 \phi_*(p(\sigma^2_p+\tau)/\sigma^2_1)\leq 2(p-2)\frac{\sigma_p^2+\tau}{\sigma_1^2+\tau}.
\end{align*}
So it suffices to show
\begin{align*}
 \phi_*(p(\sigma^2_p+\tau)/\sigma^2_1)\leq 2(p-2)\frac{\sigma_p^2+\tau}{\sigma_1^2+\tau}
\end{align*}
when $\sigma^2_1-2\sigma^2_p>0 $ and $ 0<\tau < \sigma^2_1-2\sigma^2_p$.
By the properties \ref{Bayes_phi_2} and \ref{Bayes_phi_5}, we have $ \phi_*(z)\leq \{(p-2)/p\}z$ for all $z\geq 0$.
Then 
\begin{align*}
 2(p-2) \frac{\sigma^2_p+\tau}{\sigma^2_1+\tau}
 -\phi_*(p(\sigma^2_p+\tau)/\sigma^2_1) 
 &\geq  2(p-2) \frac{\sigma^2_p+\tau}{\sigma^2_1+\tau}
 -\frac{p-2}{p}\frac{p(\sigma^2_p+\tau)}{\sigma^2_1} \\
 &=(p-2)(\sigma^2_p+\tau)\left(\frac{2}{\sigma^2_1+\tau}-\frac{1}{\sigma^2_1}\right) \\
 &=\frac{(p-2)(\sigma^2_p+\tau)}{\sigma^2_1(\sigma^2_1+\tau)}
 \left(\sigma^2_1-\tau\right) \\
 &\geq\frac{(p-2)(\sigma^2_p+\tau)}{\sigma^2_1(\sigma^2_1+\tau)}2\sigma^2_p \\
 &\geq 0,
\end{align*}
 which completes the proof.

 [Part \ref{final.thm.2}] It follows from Theorem \ref{thm:ordinary_minimax} in Appendix \ref{sec:ordinary}.

 [Part \ref{final.thm.3}] It follows from Theorem 6.4.2 of \cite{Brown-1971}.
\end{proof}

\subsection{A numerical experiment}
Let $p=10$ and 
\begin{align*}
 \bmSi=\mathrm{diag}(a^9, a^8,\dots, a, 1)
\end{align*}
for $ a=1.01, 1.05, 1.25, 1.5$. Approximately $a^9$ is $1.09, 1.55, 7.45, 38.4$, respectively.
We investigate numerical performance of
two ensemble minimax estimators of the form
\begin{align*}
 \bm{\delta}_\phi=\left(\bmI-\bmSi\frac{\phi(\|\bmx\|^2/\sigma^2_1)}{\|\bmx\|^2}\right)\bmx,
\end{align*}
where the one is the generalized Bayes estimator (GB) with
\begin{equation*}
 \phi_*(z)=z\frac{\int_0^1 \lambda^{p/2-1}\exp(-z\lambda/2)\rd\lambda }
{\int_0^1 \lambda^{p/2-2}\exp(-z\lambda/2)\rd\lambda }
\end{equation*}
and the other is the James-Stein variant (JS) with
\begin{equation*}
 \phi_{\mathrm{JS}}(z)=\frac{(p-2)z}{p-2+z}.
\end{equation*}

As in Part \ref{cor:js.2} of Theorem \ref{cor:js} and
Part \ref{final.thm.2} of Theorem \ref{final.thm},
a sufficient condition for both estimators to be ordinary minimax 
is given by
\begin{equation}\label{suff.minimax.0}
 2\left(\sum_{i=1}^p \sigma_i^4/\sigma_1^4-2\right)=2\left(\sum_{i=1}^p a^{2(i-10)}\right)\geq p-2,
\end{equation}
where the equality is attained by $a\approx 1.066$. Hence
the inequality \eqref{suff.minimax.0} is satisfied by $a=1.01, 1.05$ and is not by $a=1.25, 1.5$.

Table \ref{table:final} provides relative ordinary risk difference given by
\begin{align*}
 1-R(\bmth,\bmde_\phi)/\mathrm{tr}\bmSi
\end{align*}
at 
\begin{align}
 \bmth = m\{\mathrm{tr}\bmSi\}^{1/2} \frac{\bm{1}_{10}}{\sqrt{10}}=
 m\left\{\sum\nolimits_{i=1}^{10} a^{i-1}\right\}^{1/2}\frac{\bm{1}_{10}}{\sqrt{10}}
\end{align}
for $m=0,2, 20,40,60,80,100$.
For both estimators, we see that, for larger $m$ and $a=1.25, 1.5$, the differences are  negative, which implies that
these two estimators are not ordinary minimax for $a=1.25, 1.5$.

Table \ref{table:final_2} provides relative Bayes risk difference given by
\begin{align*}
1-  \bar{R}(\bmde,\tau)/\mathrm{tr}\bmSi
\end{align*}
for $\tau=1,5,20,40,60,80,100$.
We see that, for even $a=1.25, 1.5$, the differences are all positive, which supports
the ensemble minimaxity of two estimators.

In summary these tables support the theory presented in Theorems \ref{cor:js} and \ref{final.thm}.

\begin{table} 
\setlength{\tabcolsep}{2pt}
\caption{Ordinary Risk Difference}
\begin{center}
 \begin{tabular}{cc|ccccccc} %\toprule
& $a\backslash m$ & 0 &2& 20 & 40 & 60 & 80 & 100 \\
\midrule 
GB & $1.01$ & $0.79$ & $0.14$ & $1.7\!\times\!  10^{-3}$ & $4.8\!\times\!  10^{-4}$ & \quad $2.5\!\times\!  10^{-4}$ & \quad$1.7\!\times\!  10^{-4}$ & \quad$1.3\!\times\!  10^{-4}$ \\
& $1.05$ &  $0.75$ & $0.14$ & $1.7\!\times\! 10^{-3}$ & $4.3\!\times\!  10^{-4}$ & \quad $2.0\!\times\!  10^{-4}$ & \quad$1.2\!\times\!  10^{-4}$ & \quad$8.0\!\times\!  10^{-5}$ \\
& $1.25$ &  $0.63$ & $0.19$& $1.9\!\times\!  10^{-3}$ & $2.5\!\times\!  10^{-4}$ & $-5.6\!\times\!  10^{-5}$ & $-1.7\!\times\!  10^{-4}$ & $-2.2\!\times\!  10^{-4}$ \\
& $1.5$ &  $0.63$ & $0.27$ & $2.7\!\times\!  10^{-3}$ & $1.6\!\times\!  10^{-4}$ & $-3.0\!\times\!  10^{-4}$ & $-4.6\!\times\!  10^{-4}$ & $-5.4\!\times\!  10^{-4}$ \\
JS & $1.01$ & $0.80$ & $0.14$& $1.7\!\times\!  10^{-3}$ & $4.8\!\times\!  10^{-4}$ & \quad $2.5\!\times\!  10^{-4}$ & \quad$1.7\!\times\!  10^{-4}$ & \quad$1.3\!\times\!  10^{-4}$ \\
& $1.05$ & $0.79$ & $0.14 $& $1.7\!\times\!  10^{-3}$ & $4.3\!\times\!  10^{-4}$ & \quad $2.0\!\times\!  10^{-4}$ & \quad$1.2\!\times\!  10^{-4}$ & \quad$8.0\!\times\!  10^{-5}$ \\
& $1.25$ & $0.72$ & $ 0.19$ & $1.9\!\times\!  10^{-3}$ & $2.5\!\times\!  10^{-4}$ & $-5.6\!\times\!  10^{-5}$ & $-1.7\!\times\!  10^{-4}$ & $-2.2\!\times\!  10^{-4}$ \\
& $1.5$ & $0.71$ & $0.25$& $2.7\!\times\!  10^{-3}$ & $1.6\!\times\!  10^{-4}$ & $-3.0\!\times\!  10^{-4}$ & $-4.6\!\times\!  10^{-4}$ & $-5.4\!\times\!  10^{-4}$ \\
 \end{tabular}
\end{center}
\label{table:final}
\end{table}%

\begin{table} 
\setlength{\tabcolsep}{2pt}
\caption{Bayes Risk Difference}
\begin{center}
 \begin{tabular}{cc|ccccccc} %\toprule
& $a\backslash \tau$ & 1 & 5 & 20 & 40 & 60 & 80 & 100 \\
\midrule 
GB & $1.01$ &  $0.429$ & $0.139$& $0.039$ & $0.020$ & $0.013$ & $0.010$ & $0.008$ \\
& $1.05$ & $0.374$ & $0.144$ &$0.042$ & $0.021$ & $0.015$ & $0.011$ & $0.008$ \\
& $1.25$ & $0.105$ & $0.082$ &$0.038$ & $0.021$ & $0.014$ & $0.011$ & $0.009$ \\
& $1.5$  & $0.023$ & $0.022$ &$0.019$ & $0.014$ & $0.012$ & $0.010$ & $0.008$ \\  
  JS& $1.01$ & $0.406$& $0.137$& $0.039$ & $0.020$ & $0.014$ & $0.010$ & $0.008$ \\
& $1.05$ & $0.393$& $0.143$& $0.042$ & $0.022$ & $0.015$ & $0.011$ & $0.009$ \\
& $1.25$ & $0.122$& $0.079$&$0.034$ & $0.020$ & $0.014$ & $0.011$ & $0.009$ \\
& $1.5$  & $0.028$& $0.025$ &$0.018$ & $0.013$ & $0.010$ & $0.008$ & $0.007$ \\  
 \end{tabular}
\end{center}
\label{table:final_2}
\end{table}%

\appendix
\section{Ordinary minimaxity}
\label{sec:ordinary}
We assume
\begin{enumerate}[label=A\arabic*]
 \item \label{as.2} $0<g_i \leq 1$ for any $i$.
 \item \label{as.3} $\phi(z)\geq 0$.    
 \item \label{as.1} $\phi$ is absolutely continuous.
 \item \label{as.4} $\phi$ is monotone non-decreasing.    
\end{enumerate}
Let
\begin{align*}
 h(\bmSi,\bmG)=2\left(\frac{\sum g_i\sigma^2_i}{\max (g_i\sigma^2_i)}-2\right).
\end{align*}
Then we have a following result.
\begin{thm}\label{thm:ordinary_minimax}
 Suppose $\bmG$ satisfies \ref{as.2} and $h(\bmSi,\bmG)>0$. Then
 $\bmde_\phi$ is conventional minimax if $\phi$ satisfies \ref{as.3}, \ref{as.1}, \ref{as.4}
 and $\phi\leq h(\bmSi,\bmG)$. 
\end{thm}
\begin{proof}
The risk of $\bmde_\phi$ is expanded as
\begin{align*}
 R(\bmde_{\phi},\bmth)&=\sum_{i=1}^p E\left[\left\{\left(1-g_i\frac{\phi(z)}{z}\right)x_i-\theta_i\right\}^2\right] \\
 &=\sum_{i=1}^p E\left[(x_i-\theta_i)^2\right] + E\left[\frac{\phi(z)^2}{z^2}\sum g_i^2x_i^2 \right]
 -2\sum_{i=1}^p E\left[g_i(x_i-\theta_i)x_i\frac{\phi(z)}{z}\right].
\end{align*}
Noting $ z=\sum g_ix_i^2/\sigma_i^2$ and using \citeapos{Stein-1981} identity under Assumption \ref{as.1},
we have
\begin{equation}\label{eq:stein_1}
\begin{split}
 & R(\bmde_{\phi},\bmth)-\sum \sigma^2_i \\ &
 =  E\left[\frac{\phi(z)^2}{z^2}\sum g_i^2x_i^2 \right] 
  -2\sum_{i=1}^p E\left[g_i\sigma^2_i\left\{\frac{\phi(z)}{z}+x_i\frac{2g_ix_i}{\sigma^2_i}
 \left(\frac{\phi'(z)}{z}-\frac{\phi(z)}{z^2}\right)\right\}\right] \\
 &=  E\left[-2\sum g_i\sigma^2_i\frac{\phi(z)}{z}+\sum g_i^2x_i^2
 \left(\frac{\phi(z)^2}{z^2}-4\frac{\phi'(z)}{z}+4\frac{\phi(z)}{z^2}\right)\right].
\end{split} 
\end{equation}
By \ref{as.2}, we have
\begin{equation}\label{eq:max_g}
 \frac{z}{\sum g_i^2x_i^2}
  =\frac{\sum g_i x_i^2/\sigma^2_i}{\sum g_i^2x_i^2}
  =\frac{\sum g_i x_i^2/\sigma^2_i}{\sum \{g_i\sigma^2_i\}\{g_i x_i^2/\sigma^2_i\}}
\geq \frac{1}{\max\left(g_i\sigma^2_i\right)}.
\end{equation}
Then, by Assumptions \ref{as.3} and \ref{as.4}, and \eqref{eq:max_g}, we have 
\begin{align*}
R(\bmde_{\phi},\bmth)-\sum \sigma^2_i
 \leq -E\left[\sum g_i^2x_i^2\frac{\phi(z)}{z^2}\left\{h(\bmSi,\bmG) -\phi(z)
 \right\} \right], 
\end{align*}
which completes the proof.
\end{proof}

\end{document}